\documentclass[a4paper,english]{article}

\usepackage{amssymb,amsthm, amsmath,amsfonts}
\usepackage{dsfont}
\usepackage{graphicx}
\usepackage{enumerate}
\usepackage{authblk}
\usepackage{subfigure}
\usepackage{microtype}

\newtheorem{theorem}{Theorem}[section]

\newtheorem{lemma}[theorem]{Lemma}

\theoremstyle{definition}

  \newtheorem{example}[theorem]{Example}

\newcommand{\dG}{\mathbb{G}}
\newcommand{\dH}{\mathbb{H}}
\newcommand{\dP}{\mathbb{P}}
\newcommand{\calL}{\mathcal{L}}
\newcommand{\calM}{\mathcal{M}}
\newcommand{\calR}{\mathcal{R}}
\newcommand{\calG}{\mathcal{G}}
\newcommand{\Z}{\mathbb{Z}}
\newcommand{\N}{\mathbb{N}}
\newcommand{\id}{\mathds{1}}

\begin{document}

\date{}

   \setcounter{Maxaffil}{3}
  
    \title{
Insertion-tolerance and repetitiveness of random graphs
}
    
    \renewcommand\footnotemark{}
    
    \author[a,b]{Fernando Alcalde Cuesta}
 
    \author[a,c,d]{\'Alvaro Lozano Rojo} 

     \author[a,b]{Ant\'on C. V\'azquez Mart\'{\i}nez}

     \affil[a]{\small ~GeoDynApp - ECSING Group (Spain)}
      \affil[b]{\small ~Departamento de Matem\'aticas, Universidad de Santiago de Compostela, 
      E-15782 Santiago de Compostela (Spain)}
    
      \affil[c]{\small ~Centro Universitario de la Defensa, 
        Academia General Militar, Ctra. Huesca s/n. \newline E-50090 Zaragoza (Spain)
      }
     \affil[d]{\small ~Instituto Universitario de Matem\'aticas y Aplicaciones, 
        Universidad de Zaragoza (Spain)
      }

\maketitle

\vspace{-2em}

\begin{abstract}

\noindent
 Bond percolation on Cayley graphs provides examples of random graphs. Other 
  examples arise from the dynamical study of proper repetitive subgraphs of 
  Cayley graphs. In this paper we demonstrate that these two families have 
  mutually singular laws as a corollary of a general lemma about countable 
  Borel equivalence relations on first countable Hausdorff spaces. 
\medskip 

{\footnotesize

}
 \end{abstract}

\section{Introduction}

The notion of \emph{random rooted graph} (as described by D.~Aldous and 
R.~Lyons in~\cite{AL}) arises in substance from the paper \cite{LPP} of 
R.~Lyons, R.~Pemantle, and Y.~Peres on Galton-Watson trees. Bernoulli bond 
percolation on a Cayley graph $\dG$ provides the basic example of random rooted 
graph, which is obtained by keeping each edge with constant probability $p$ 
independently to other edges, see Figure~\ref{fig:cluster}. This kind of random 
graphs (and other obtained by bond percolation on Cayley graphs or unimodular 
transitive graphs) enjoy the important property of \emph{insertion-tolerance} 
introduced by R. Lyons and O. Schramm in Definition 3.2 of \cite{LS}: the 
measure of a nonnull Borel set after adding an edge is nonnull.
\medskip
\begin{figure}
  \centering
\subfigure[Insertion-tolerant random graph]{
  \includegraphics[height=2.3in]{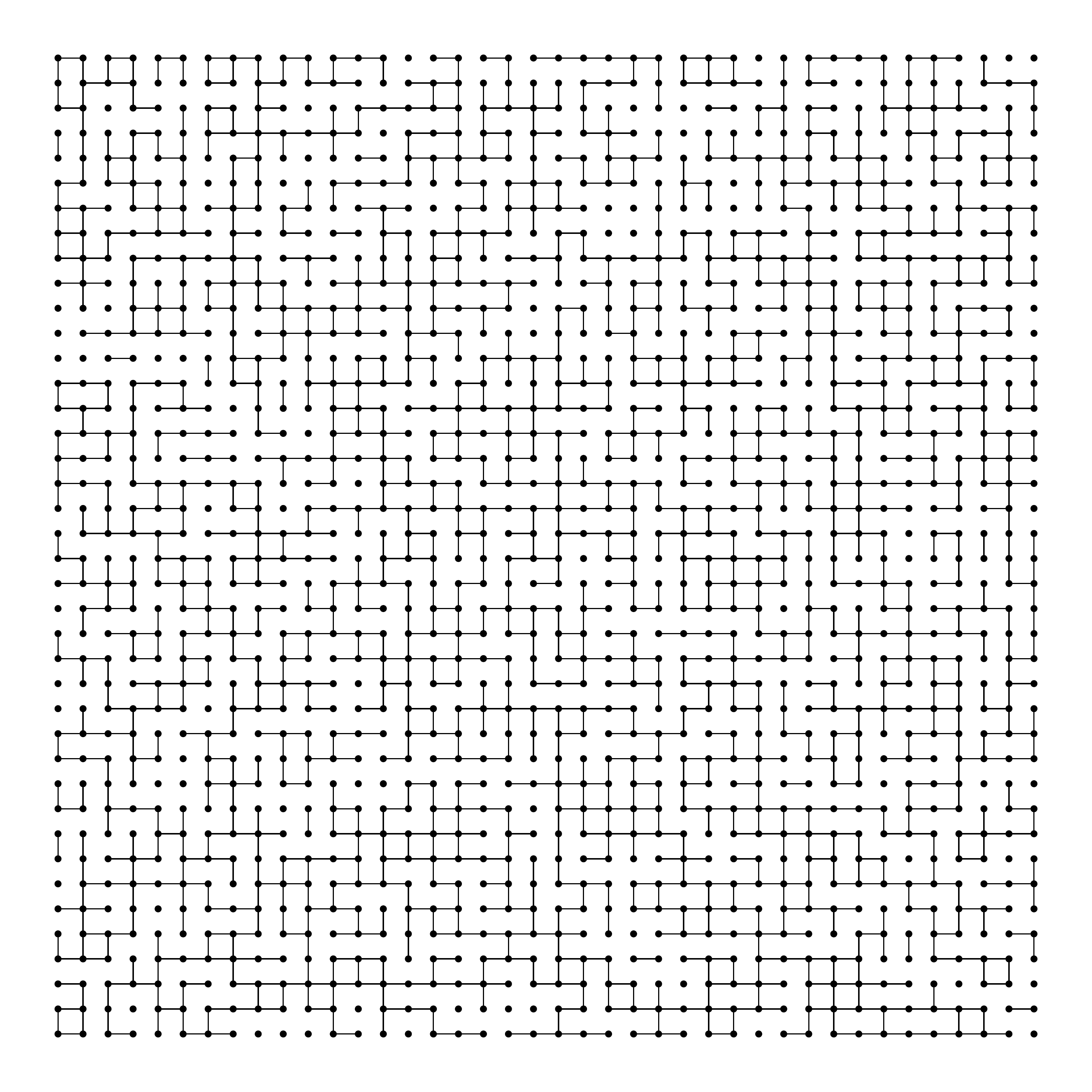}
    \label{fig:cluster}
  }
\subfigure[Repetitive random graph]{
  \includegraphics[height=2.2in]{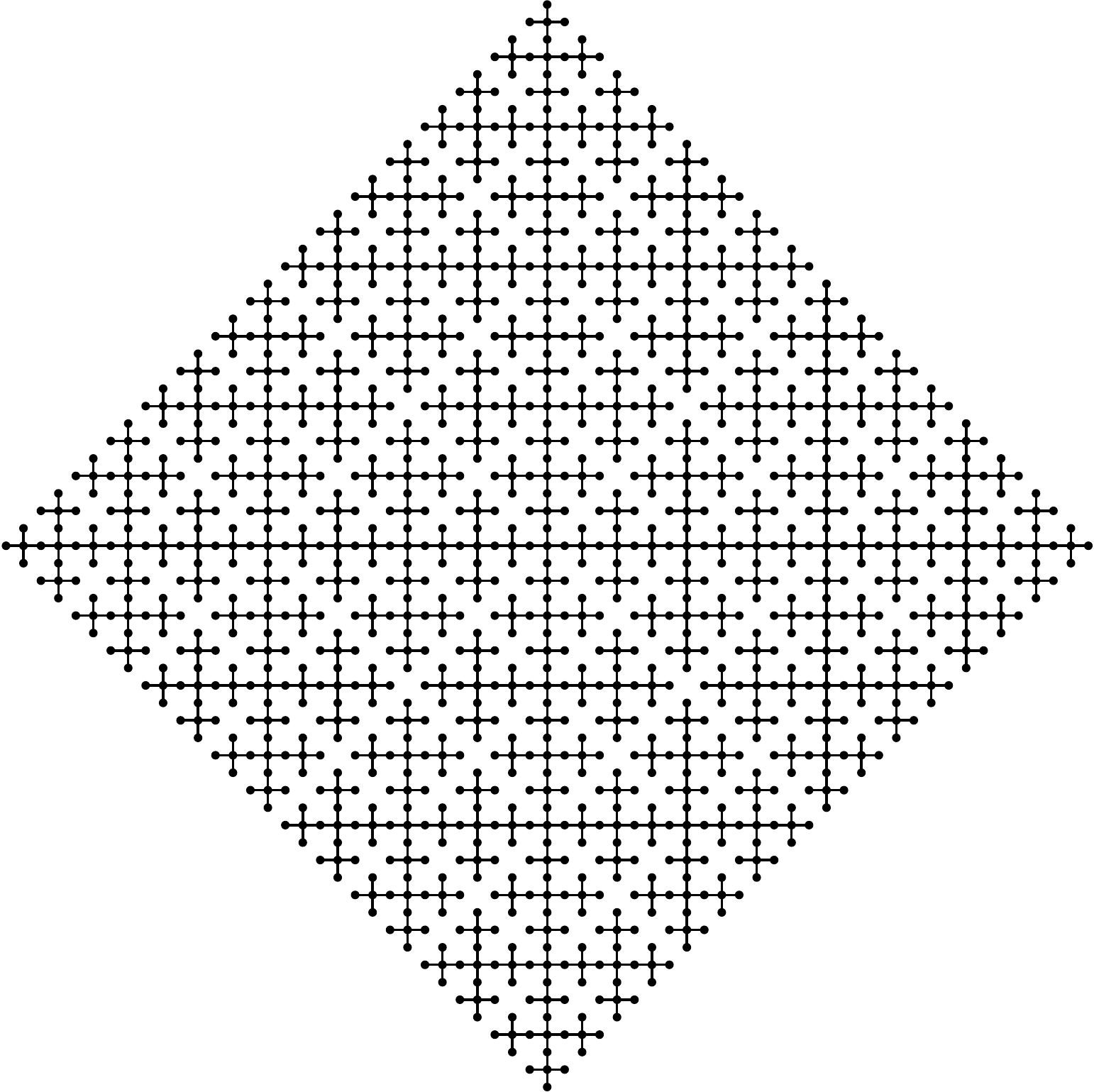}
  \label{fig:Kenyontree}
  }
  \caption{Random subgraphs of the Cayley graph of $\Z^2$.}
\end{figure}

However, this property does not hold for other examples of random graphs 
arising as orbit closures of repetitive subgraphs of $\dG$ in the 
Gromov-Hausdorff space (described in \cite{Gh}; see also \cite{ALM} and 
\cite{B}), see Figure~\ref{fig:Kenyontree}. The \emph{repetitiveness} of a 
subgraph of $\dG$ is equivalent to the minimality of its orbit closure. 
\medskip 

In this paper, we show that the two properties above cannot occur 
simultaneously. More precisely, any random subgraph of a Cayley graph $\dG$ 
with  insertion-tolerant, ergodic in restriction to the set of infinite states 
and quasi-invariant law and the orbit closure of any proper repetitive subgraph 
of $\dG$ are mutually singular. 

\section{A lemma for Borel equivalence relations}
\label{SBorel}

Let $X$ be a Borel space and let $\calR \subset X \times X$ be a countable 
Borel equivalence relation. For each $x \in X$, we define 
$\calR [x] = \{ \,y \in X \mid (x,y) \in \calR \,\}$,  and similarly for each 
Borel set $B \subset X$, 
\[
  \calR [B] = \bigcup_{x \in B} \calR [x].
\]
A Borel probability measure $\mu$ on $X$ is \emph{$\calR$-invariant} if 
$\varphi_\ast (\mu|_A) = \mu|_B$ for any partial transformation 
$\varphi:A\to B$ of $\calR$ (i.e. $\varphi$ is a measurable bijection whose 
graph is contained in $\calR$), where $\mu|_A$ and $\mu|_B$ are the measures 
restricted to $A$ and $B$ respectively. If only $\mu$-null sets are preserved, 
$\mu$ is \emph{$\calR$-quasi-invariant} and $\calR$ is 
\emph{$\mu$-nonsingular}. The measure $\mu$ is \emph{$\calR$-ergodic} if either 
$\mu(\calR [B]) = 0$ or $\mu(\calR [B]) = 1$ for every Borel set $B \subset X$. 
\medskip 

If $X$ is also equipped with a topology, a closed subset $Z \subset X$ is 
\emph{$\calR$-minimal} if every class $\calR[x]$ is dense in $Z$. Under some 
additional assumptions, we have the following general lemma: 

\begin{lemma}
  \label{keylemma}
  Let $X$ be a first countable Hausdorff space and let $\calL\subset X\times X$ 
  be a countable Borel equivalence relation. Let $\calR$ be a nonsingular 
  equivalence subrelation of $\calL$ equipped with an ergodic quasi-invariant 
  probability measure $\mu$. Suppose that there exists a point $x$ in the 
  support of $\mu$ such that $\calL[x] = \{x\}$. If $Z$ is a closed 
  $\calL$-minimal subset of $X$, then either $Z = \{x\}$ or $\mu(Z) = 0$. 
\end{lemma}

\begin{proof} 
  Since $X$ is a first countable Hausdorff space and $x$ belongs to the support 
  of $\mu$, there is a decreasing sequence of open neighborhoods $U_n$ of $x$ 
  such that
  \[
    \bigcap_{n \in \N} U_n = \{x\}
      \quad \text{and} \quad 
    \text{$\mu(U_n)>0$ for all $n \in \N$.} 
  \]
  As $\mu$ is $\calR$-quasi-invariant and $\calR$-ergodic, $\mu(\calR[U_n])=1$ 
  for all $n \in \N$ and hence
  \[
    Y = \bigcap_{n \in \N}\calR[ U_n]
  \]
  is also a conull Borel set. Let $Z$ be a closed $\calL$-minimal subset of $X$ 
  and assume that there is a point $z \in Z \cap Y$. Therefore, since 
  $\calR \subset \calL$, the point 
  \[
    z \in Z \cap \calR[ U_n] \subset Z \cap \calL[U_n]
  \] 
  for all $n\in\N$. Thus, for each $n\in\N$, we can find a point 
  $z_n\in Z\cap U_n$. The sequence of points $z_n \in Z$ converges to $x$, but 
  as $Z$ is closed, it follows that $x \in Z$. By the $\calL$-minimality of 
  $Z$, we conclude that $Z =\{x\}$. Finally, if $Z \cap Y = \emptyset$, then 
  $\mu(Z) = 0$ since $Y$ is conull. 
\end{proof}

Actually, assuming $X$ is compact, we can replace $\{x\}$ by any 
$\calL$-minimal set $\calM$ contained in the support of $\mu$. If there is a 
point $z \in Z \cap Y$,  we still obtain a sequence of points 
$z_n\in X\cap U_n$. By compactness, passing to a subsequence if necessary, we 
can assume that $z_n$ converges to point $x \in \calM$. But as before, since 
$Z$ is closed, the limit point $x \in Z$ and hence $Z = \calM$ by minimality.

\section{Random subgraphs of Cayley graphs}
\label{Srandomsubgraphs}

Let $G$ be a countable group $G$ with a finite and symmetric generating set 
$S\subset G$. The \emph{Cayley graph} $\dG = (V,E)$ associated to $(G,S)$ is 
defined by $V = G$ and $E = \{\,(g,h)\in G\times G\mid hg^{-1}\in S\,\}$. It 
has a natural right $G$-action by graph automorphisms that combines the natural 
right action of $G$ on itself and the right diagonal action of $G$ on 
$E\subset G\times G$. Let $2^E$ the power set of all subsets of $E$, equipped 
with the product topology. This is a compact metrizable space, so in particular 
first countable and Hausdorff. It also have a natural left $G$-action by 
homeomorphisms, which is induced by the natural right $G$-action on $E$, namely 
\[
  g.\omega = \omega g^{-1}
\]
for all $g \in G$ and all $\omega \in 2^E$. Note that $E$ is a fixed point for 
this action. Given any finite and symmetric subset $F \subset E$, we consider 
the open set 
\[
  U_F = \{\,\omega \in 2^E \mid F \subset \omega\,\}.
\]
and thus we have that
\[
  \{E\} = \bigcap_{n \in \N} U_{F_n}
\]
for any increasing exhaustion $\{F_n\}$ of $E$ by finite and symmetric subsets.
\medskip 

For each $g \in G$ and each $\omega \in 2^E$, we denote by $C_g(\omega)$ the 
connected component or \emph{cluster} of the graph $(G,\omega)$ that contains 
$g$. If we take $g$ equal to the identity $\id$, the map associating to each 
element $\omega \in 2^E$ the cluster $C_\id(\omega) \subset \dG$ is a 
continuous map. It takes values in the \emph{Gromov-Hausdorff space} $\calG$ 
formed of all connected subgraphs $\dH$ of $\dG$ containing $\id$ (see 
\cite{ALM} and \cite{Gh}). In fact, this space is endowed with the ultrametric
\[
  d(\dH,\dH') = 1/\exp(\sup\{\,r \geq 0 \mid B_{\dH}(\id,r) 
              = B_{\dH'}(\id,r)\,\}), 
\]
where $B_{\dH}(\id,r)$ is the combinatorial closed ball of radius $r$ centered 
at $\id$. 
\medskip 

The \emph{orbit equivalence relation} $\calR^G$ on $2^E$ is defined by 
\[
  \calR^G = \{\,(\omega,\omega') \in 2^E\times 2^E 
        \mid\exists\,g\in G:\omega'=g.\omega \,\}, 
\]
and the \emph{cluster equivalence relation}s $\calR$ on $2^E$ by 
\[
  \calR = \{\,(\omega,\omega') \in 2^E \times 2^E \mid 
                \exists \, g \in C_\id(\omega) : \omega' = g.\omega \,\}. 
\]
On the quotient $\calG$, we have another natural equivalence relation, abusively denoted by $\calR$, which is defined by 
\[
  \calR = \{\,(\dH,\dH') \in \calG \times \calG 
              \mid \exists \, g \in G : \dH' = g.\dH =\dH'g^{-1} \,\}
\]
and induced by the equivalence relation $\calR$ on $2^E$. These two compatible 
actions are the orbital equivalence relations defined by the left actions of the 
pseudogroups generated by the local transformations $\omega \mapsto g.\omega$ 
and $\dH \mapsto g.\dH$ defined on the open subset of $2^E$ of all sets 
$\omega$ such that $g\in C_\id(\omega)$ and the open subset of $\calG$ of all 
subgraphs $\dH$ containing the vertex $g$ respectively. 
\medskip 

Let $\mu$ be a probability measure on $X=2^E$ that makes the cluster 
equivalence relation $\calR$ nonsingular. Assume that:
\begin{itemize}

  \item[(i)] The $\calR$-invariant set 
    $X_\infty = \{\,\omega \in 2^E \mid card(\calR[\omega]) = \infty \,\}$ 
    has positive measure, and the restriction of $\mu$ to this set is 
    $\calR$-ergodic.

  \item[(ii)] The measure $\mu$ is \emph{insertion-tolerant}, i.e. for each finite 
    and symmetric set $F \subset E$, the map $i_F : 2^E \to 2^E$ given by 
    \[
      i_F(\omega) = \omega \cup F
    \]
    preserves all Borel sets of positive $\mu$-measure. 

\end{itemize}
Since $\mu(X_\infty)>0$, we can replace $\mu$ with itself conditioned to 
$X_\infty$, which is still denoted abusively by $\mu$. By our second 
assumption, if $\{F_n\}$ is an increasing exhaustion of $E$, the open sets 
$i_{F_n}(X_\infty) = U_{F_n} \cap X_\infty$
have positive $\mu$-measure for every $n$. It follows that all the open 
neighborhoods $U_{F_n}$ of $E$ have also positive $\mu$-measure and hence $E$ 
belong to the support of $\mu$. Finally, since $E$ is a fixed point of the 
$G$-action on $X$ and $C_\id(E) = \dG$, we conclude that $\calR[E] = \{E\}$. 
Lemma~\ref{keylemma} applied to the first countable Hausdorff space $X = 2^E$ 
equipped with the equivalence relation $\calL = \calR$ and the $G$-fixed point 
$x = E$ yields:

\begin{theorem}
  \label{thm:maintheorem1}
  Let $\mu$ be a probability measure on $X=2^E$. If $\mu$ is quasi-invariant 
  and ergodic with respect to $\calR$ and insertion tolerant, then any closed 
  $\calR$-minimal set $Z \subset X$ is either $Z = \{E\}$ or $\mu(Z) = 0$. \qed
\end{theorem}

In this result, we can replace the configuration space $X=2^E$ with the 
Gromov-Hausdorff space $\calG$ so that a probability measure $\mu$ makes 
$\calG$ a \emph{random subgraph} of the Cayley graph $\dG$. If $\mu$ is 
quasi-invariant with respect to $\calR$, we say the random graph $(\calG,\mu)$ 
is \emph{nonsingular}. Let $\calG_\bullet$ be the locally compact metrizable 
space of the isomorphism classes of locally finite connected rooted graphs (see 
\cite{BS}). Using the natural continuous map from $\calG$ to $\calG_\bullet$ sending 
each graph $\dH$ on its isomorphism class $[\dH]$, we can interpret $(\calG,\mu)$ as a \emph{random graph} in 
the sense of \cite{AL}. On the other hand, as pointed out in the introduction, 
it is well known (see for example \cite{ALM} and \cite{B}) that the closed 
$\calR$-minimal subsets of $\calG$ are in one-to-one correspondence with the 
orbit closures of repetitive subgraphs of $\dG$. Now, we can state an 
equivalent version of the previous theorem: 

\begin{theorem}
  \label{thm:maintheorem2}
  Any nonsingular random subgraph of a Cayley graph $\dG$ whose law is 
  insertion-tolerant and ergodic in restriction to the set of its infinite 
  states and any orbit closure of a proper repetitive subgraph of $\dG$ are 
  mutually singular. \qed
\end{theorem}

\section{Examples}

In this section, we assemble some examples to which Theorems~\ref{thm:maintheorem1}~and~\ref{thm:maintheorem2} apply. 

\begin{example} \label{ex:Bernoulli} 
  \emph{Bernoulli bond percolation} on a Cayley graph $\dG$ provides the basic 
  example of random graph, which is obtained by keeping each edge with constant 
  probability $p$ independently to other edges. More precisely, the power set 
  $2^E$ is equipped with the Bernoulli measure $\mu$ with constant survival 
  parameter $p$ (which means that $\mu$ is the product of the measure on 
  $\{0,1\}$ assigning probabilities $1-p$ and $p$ to $0$ and $1$ respectively). 
  In the supercritical phase $p_c < p <1$, if we denote by $\dP$ the law that 
  governs the clusters $C_\id(\omega)$, the set of infinite clusters has 
  positive $\dP$-measure. According to the cluster indistinguishability theorem 
  proved by R.~Lyons and O.~Schramm in \cite{LS} (see also \cite{GL}), the 
  measure $\dP$ conditioned to the set of infinite clusters is $\calR$-ergodic 
  and insertion-tolerant. Then the orbit closure of any proper repetitive 
  subgraph $\dH$ of $\dG$ is $\dP$-null. 
  \end{example}

\begin{example} \label{ex:percolation} 
Since Theorem 3.3 of \cite{LS} is valid for any insertion-tolerant 
  $\calR$-invariant probability measure $\mu$ on $2^E$, 
  Theorem~\ref{thm:maintheorem2} also applies to random graphs obtained by this 
  kind of $G$-invariant percolation. Recall that a \emph{$G$-invariant bond 
  percolation process} on $\dG$ is given by a measure preserving $G$-action on 
  a standard Borel probability space $(X,\mu)$ together with a $G$-equivariant 
  Borel map $\pi:X\to 2^E$ (see \cite{G}). In the case under consideration, 
  the map $\pi$ is the identity $id$, but the law of the process is not longer 
  the Bernoulli measure. However, the cluster indistinguishability theorem also 
  holds for some $G$-invariant percolation processes with $\pi\neq id$, like 
  the \emph{percolation process with scenery} introduced in \cite[Remark 3.4]{LS} (see also \cite[Proposition 6]{GL}).
\end{example}

\begin{example} \label{ex:unimodular} 
Let $\dG = (V,E)$ be a locally finite graph and let $G$ be a unimodular 
  closed group of automorphisms of $\dG$ acting transitively on $V$. Then the 
  proof of Theorem~\ref{thm:maintheorem1} extends to this case. Thus, if $\mu$ 
  is an insertion tolerant, $\calR$-ergodic and $\calR$-invariant probability 
  measure on $2^E$, any closed $\calR$-minimal subset $Z$ of $2^E$ is either 
  $Z = \{E\}$ or $\mu(Z)=0$. But Theorem 3.3 of \cite{LS} is also valid for any 
  $G$-invariant insertion-tolerant bond percolation process on $\dG$, and hence 
  Theorem~\ref{thm:maintheorem2} applies to all unimodular random graphs 
  obtained by this kind of percolation.
\end{example}

\section*{Acknowledgements} 

The Lemma~\ref{keylemma} and its proof were provided by the referee based on 
the first version of the paper. Section~\ref{Srandomsubgraphs} also benefited 
from their suggestions and comments. The authors were supported by Spanish 
Excellence Grant  MTM2013-46337-C2-2-P, Galician Grant GPC2015/006 and the 
European Regional Development Fund. Second author was also supported by the 
European Social Fund and Diputaci\'on General de Arag\'on (Grant E15 
Geometr\'ia).

\bigskip 

\noindent
\textit{E-mail addresses:} \par 
 \addvspace{\medskipamount}
 
{\footnotesize
\noindent
Fernando Alcalde Cuesta:~\texttt{fernando.alcalde@usc.es} \par  \addvspace{\smallskipamount}
\noindent
\'Alvaro Lozano Rojo:~\texttt{alvarolozano@unizar.es}  \par  
\noindent
Ant\'on C. V\'azquez Mart\'{\i}nez:~\texttt{anton.vazquez@yahoo.es}
}

\end{document}